\documentclass[12pt,paper]{amsart}
\usepackage{amsmath}
\usepackage{amsxtra}
\usepackage{amscd}
\usepackage{amsthm}
\usepackage{amsfonts}
\usepackage{amssymb}
\usepackage{eucal}
\usepackage{mathrsfs}
\usepackage{hyperref}

\def\F{{\mathbb F}}
\def\G{{\mathbb G}}

\def\P{{\mathbb P}}

\def\Z{{\mathbb Z}}

\newtheorem{theorem}{Theorem}[section]

\newtheorem{lemma}[theorem]{Lemma}
\newtheorem{remark}{Remark}

\newtheorem{ex}{Example}
\newtheorem{proposition}[theorem]{Proposition}
\newtheorem{corollary}[theorem]{Corollary}

\title[Secant Loci and Simple Linear Projections]{On secant loci and simple linear projections of some projective varieties}
\author[E. Park]{Euisung Park}
\address {Euisung Park : Department of Mathematics, Korea University, Seoul 136-701, Republic of Korea} \email{euisungpark@korea.ac.kr}
\thanks{``This work was supported by the Korea Research Foundation Grant funded by the Korean Government(MOEHRD)." (KRF- 2006-214-C00002)}

\begin{document}

\thispagestyle{empty} \maketitle

\begin{abstract}
In this paper, we study how simple linear projections of some projective
varieties behave when the projection center runs through the ambient space. More
precisely, let $X \subset \P^r$ be a projective variety satisfying Green-Lazarsfeld's property
$N_p$ for some $p \geq 2$, $q \in \P^r$ a closed point outside of $X$, and $X_q := \pi_q (X) \subset \P^{r-1}$ the projected image of $X$ from $q$. First, it is shown that the secant locus $\Sigma_q (X)$ of $X$
with respect to $q$, i.e. the set of all points on $X$ spanning secant lines passing through
$q$, is either empty or a quadric in a subspace of $\P^r$. This implies that the finite
morphism $\pi_q  : X \rightarrow X_q$ is birational. Our main result is that cohomological and
local properties of $X_q$ are precisely determined by $\Sigma_q (X)$. To complete this result,
the next step should be to classify all possible secant loci and to decompose the
ambient space via the classification of secant loci. We obtain such a decomposition
for Veronese embeddings and Segre embeddings. Also as an application of the main
result, we study cohomological properties of low degree varieties.
\end{abstract}

\tableofcontents \setcounter{page}{1}

\section{Introduction}
Let $X \subset \P^r$ be an $n$-dimensional nondegenerate irreducible projective variety over an algebraically closed field $K$. For a closed point $q \in \P^r$ outside of $X$, consider the subvariety
\begin{equation*}
X_q = \pi_q (X) \subset \P^{r-1}
\end{equation*}
where $\pi_q : X \rightarrow \P^{r-1}$ is the linear projection of $X$ from $q$. In this situation, it is
necessary to understand how the relative location of $q$ with respect to $X$ effects $X_q$, or
how the geometric and algebraic properties of $X_q$ behave when the projection center
runs through the ambient space. In relation to this problem, a natural geometric
approach is to investigate the \textit{secant locus} $\Sigma_q (X)$ of $X$ with respect to $q$, i.e. the set of all points on $X$ spanning secant lines passing through $q$. Here $\Sigma_q (X)$ is also called
the entry locus of $X$ with respect to $q$ in the sense of \cite{AR2}.

The main goal of this paper is to study the problem outlined above for a class of
varieties with a simple syzygetic structure. More precisely, the problem under consideration is \\

\begin{enumerate}
\item[$(\star)$] Let $X \subset \P^r$ be a nondegenerate irreducible projective variety satisfying Green-
Lazarsfeld's property $N_p$ for some $p \geq 2$. For a closed point $q \in \P^r$ outside
of $X$, describe the relation between the position of $q$ with respect to $X$ and
cohomological and local properties of $X_q := \pi_q (X) \subset \P^{r-1}$. \\
\end{enumerate}

The cohomological properties include the calculation of the Hilbert function, Hilbert
polynomial, Hartshorne-Rao module, Castelnuovo-Mumford regularity, arithmetic
depth and projective dimension, etc, which are governed by the cohomology groups
$H^i (\P^{r-1}, \mathcal{I}_{X_q} (j))$ for $i \geq 0$ and $j \in \Z$. The local properties include the descriptions of the locus of smooth points, normal points, $S_2$-points and Cohen-Macaulay points.

{\bf Theorem 3.3} shows that if $X \subset \P^r$ satisfies property $N_p$ for some $p \geq 2$, then the cohomological and local properties of $X_q \subset \P^{r-1}$ can be precisely determined by $\Sigma_q (X)$. First, $\Sigma_q (X)$ is shown to be either empty or else a quadric in a subspace of $\P^r$. This implies that the finite morphism $\pi_q : X \rightarrow X_q$ is birational. This reproves a version of a result of Vermeire \cite{V} on the linear system of quadrics through a variety satisfying property $N_2$. With regard to the cohomological aspects, $X_q$ is shown to be $j$-normal for all $j \geq 2$, to satisfy property $N_{3,p}$, and that the Castelnuovo-Mumford regularity of $X_q$ is equal to $\mbox{max} \{ 3, \mbox{Reg}(X)\}$. This extends results of Kwak-Park \cite{KP} and Choi-Kwak-Park \cite{CKP} on the isomorphic projections of a variety satisfying property $N_p$. The most interesting connection between $\Sigma_q (X)$ and $X_q \subset \P^{r-1}$ is illustrated by the following two formulae:
\begin{equation*}
\begin{CD}
h^0 (\P^{r-1},\mathcal{I}_{X_q /\P^{r-1}} (2)) & \quad = \quad & h^0 (\P^r,\mathcal{I}_{X/\P^r} (2))+s-r \\
\mbox{depth}(X_q) & = & \mbox{min}\{\mbox{depth}(X),s+2\}
\end{CD}
\end{equation*}
where $s$ is used to denote the dimension of $\Sigma_q (X)$. The depth formula is proved under the cohomological assumption
\begin{equation*}
H^i (X,\mathcal{O}_{X} (j))=0 \quad \mbox{for $1 \leq i \leq n-1$ and all $j \leq -i$},
\end{equation*}
which holds if $X$ is a smooth variety over an algebraically closed field of characteristic
zero or $X$ is arithmetically Cohen-Macaulay. The birational morphism $\pi_q : X \rightarrow X_q$ fails to be isomorphic exactly along $\mbox{Sing}(\pi_q)$. In {\bf Theorem 3.3.(7) and (8)}, the local
properties of $X_q$ at $\pi_q (x) \in \mbox{Sing}(\pi_q)$ are described comparing with those of $X$ at $x \in \Sigma_q (X)$. {\bf Theorem 3.3} generalizes a result of M. Brodmann and P. Schenzel \cite{BS} on arithmetic properties of a birational projection of a variety of minimal degree from a closed point to a variety satisfying property $N_2$. We give examples showing that these results are sharp in various senses.

Let us recall the following examples of varieties satisfying property $N_2$ to see the wide range of possible applications: Veronese embedding of projective spaces, Segre varieties, varieties of minimal degree, Del Pezzo varieties of degree at least $5$, every linearly normal embedding of Grassmannians, all embeddings given by a sufficiently ample line bundle and embeddings given by suitable adjoint line bundles satisfy property $N_2$.

To complete the result in {\bf Theorem 3.3} for a given variety $X \subset \P^r$ which satisfies property $N_2$, the next step should be to classify all possible secant loci and to decompose the ambient space via the classification of secant loci. More precisely, let
the $s$-th \textit{secant strata} of $X$, denoted $SL_s (X)$, be the set of all closed points $q \in \P^r$ satisfying $\mbox{dim} \Sigma_q (X) = s$. The ambient space is expressed as
\begin{equation*}
\P^r = \bigcup_{-1 \leq s \leq n} SL_s (X),
\end{equation*}
which we call the \textit{secant stratification} of $X \subset \P^r$.

In $\S~ 4$, the basic properties of secant loci are investigated. Also we obtain the secant stratification of Veronese embeddings and Segre varieties, which depends heavily on several intrinsic and extrinsic properties of them.

In $\S~ 5$, we apply our results to the study of cohomological properties of low degree varieties. More precisely, recall that the degree and the codimension of a nondegenerate irreducible projective varieties $X \subset \P^r$ satisfies the relation
\begin{equation*}
\mbox{deg} (X) = \mbox{codim}(X,\P^r)+k \quad \mbox{for some $k \geq 1$}.
\end{equation*}
A classical problem in algebraic geometry is to classify and to find a structure theory of projective varieties having small values of $k$. We reprove Theorem A in \cite{HSV} and Theorem 1.3 in \cite{BS} on cohomological properties of varieties of almost minimal
degree, i.e. varieties whose degree exceeds the codimension by precisely $2$.

\begin{remark}
Let $X \subset \P^r$ be a nondegenerate projection variety and let $q \in \P^r$ be
a closed point. We can consider three kinds of simple linear projections of $X$ with
respect to the location of $q$:\\

\begin{enumerate}
\item[(a)] an isomorphism outer projection of $X$ with $q$ outside of $X$
\item[(b)] a singular outer projection of $X$ with $q$ outside of $X$
\item[(c)] an inner projection of $X$ with $q$ contained in $X$ \\
\end{enumerate}

\noindent Recently, we have been interested in the effect of property $N_p$ of $X \subset \P^r$ to the
algebraic and geometric behavior of the projected image $X_q := \pi_q (X) \subset \P^{r-1}$ when
$q$ runs through the ambient space. In relation to this problem, an answer has been
provided in \cite{CKP} and \cite{KP} when $\pi_q : X \rightarrow \P^{r-1}$ is an embedding and in \cite{CKK} when
$X$ is smooth and $q$ is in $X$ and outside of the union of all lines contained in $X$. In
particular, Theorem 1.1, Proposition 2.3 and Corollary 2.4 in \cite{CKK} imply that the
two formulae in {\bf Theorem 3.3.(2) and (5)} still hold when $X$ is smooth and $q$ is in $X$
and outside of the union of all lines contained in $X$.
\end{remark}

{\bf Acknowledgements.} This paper was started when I was conducting Post Doctorial research at Institute of Mathematics in University of Zurich. I would like to thank Professor Markus Brodmann for many useful discussions and especially for his advise about {\bf Theorem 3.3.$(7)$ and $(8)$}.

\section{Preliminaries}
In this section we recall a few notation, which we use throughout this paper. For a nondegenerate irreducible projective variety $X \subset \P^r$, let $\mathcal{I}_X$ be the sheaf of ideals of $X$ and let $\F_{\bullet}$ be a minimal free resolution of the homogeneous ideal $I_X$ of $X$ over the homogeneous coordinate ring $R$ of $\P^r$:
\begin{equation*}
\F_{\bullet} : 0 \rightarrow F_r \rightarrow \cdots \rightarrow F_i \rightarrow \cdots \rightarrow F_1 \rightarrow I_X \rightarrow 0
\end{equation*}
where $F_i = \bigoplus_{j \in \Z} R(-i-j)^{\beta_{i,j}}$.\\

$(2.1)$ For $j \geq 1$, $X$ is said to be $j$-normal if the natural map
\begin{equation*}
H^0 (\P^r,\mathcal{O}_{\P^r} (j)) \rightarrow H^0 (X,\mathcal{O}_X (j))
\end{equation*}
is surjective, or equivalently, $H^1 (\P^r,\mathcal{I}_X (j))=0$. $X$ is \textit{linearly normal} if it is $1$-normal. $X$ is \textit{projectively normal} if it is $j$-normal for all $j \geq 1$.\\

$(2.2)$ For $p \geq 1$, $X$ is said to satisfy \textit{property $N_p$} if it is projectively normal and $\beta_{i,j}=0$ for $1 \leq i \leq p$ and all $j \geq 2$. Thus when $X$ is projectively normal, it satisfies property $N_1$ if $I_X$ is generated by quadrics. Also for $p \geq 2$, $X$ satisfies property $N_p$ if
it satisfies property $N_1$ and the $k$-th syzygies among the quadrics are generated by linear syzygies for all $1 \leq k \leq p-1$.\\

$(2.3)$ For some $d \geq 2$ and $p \geq 1$, $X$ is said to satisfy property $N_{d,p}$ if $\beta_{i,j}=0$ for $1 \leq i \leq p$ and all $j \geq d$, or equivalently, $I_X$ is generated by forms of degree at most $d$ and the first $(p-1)$ steps of a minimal free resolution of $(I_X)_{\geq d}$ is linear. Thus property $N_{2,p}$ coincides with property $N_p$ without projective normality(cf. \cite{EGHP}).\\

$(2.4)$ We say that $X$ is $m$-regular if $\beta_{i,j}=0$ for all $j \geq m$. It is well-known that $X$ is $m$-regular if and only if $H^i (\P^r,\mathcal{I}_X (m-i))=0$ for all $i \geq 1$(e.g. \cite{EG}). Therefore if $X$ is $m$-regular, then it is $j$-normal for all $j \geq m-1$ and it satisfies property $N_{m,p}$ for all $p \geq 1$. $\mbox{Reg}(X)$ denotes $\mbox{min}\{m \in \Z~|~X \mbox{ is $m$-regular}\}$.\\

$(2.5)$  \textit{The arithmetic depth} of $X$, denoted by $\mbox{depth}(X)$, is defined to be the depth of $R/I_X$ as an $R$-module. It is cohomologically characterized as follows:
\begin{equation*}
\mbox{depth}(X) = \mbox{min} \{ ~ i \geq 1 ~|~ \bigoplus_{j \in
\Z} H^i (\P^r,\mathcal{I}_X (j)) \neq 0 ~ \}
\end{equation*}
Therefore $1 \leq \mbox{depth}(X) \leq \mbox{dim}~X +1$. The length of the minimal free resolution of $R/I_X$ is called \textit{the projective dimension} of $X$ and is denoted by $\mbox{pd}(X)$. By the Auslander-Buchsbaum theorem,
\begin{equation*}
\mbox{pd}(X) = r+1 - \mbox{depth}(X).
\end{equation*}
We say that $X$ is \textit{arithmetically Cohen-Macaulay} if $\mbox{depth}(X)=\mbox{dim}~X +1$, or equivalently, $\mbox{pd}(X)=\mbox{codim}(X,\P^r)$. Note that if $X \subset \P^r$ is not linearly normal then $\mbox{depth}(X)=1$ and $\mbox{pd}(X)=r$. One can find the details in \cite{E}.\\

$(2.6)$ Let $\mathcal{L}=\mathcal{O}_X (1)$ and let $n=\mbox{dim}~X$. The function
\begin{equation*}
k \quad \longmapsto \quad \chi_{(X,\mathcal{L})} (k) := \sum_{i=0} ^{n} (-1)^i h^i (X,\mathcal{L}^{\otimes k})
\end{equation*}
is a polynomial of degree $n$, the so called \textit{Hilbert polynomial} of $(X,\mathcal{L})$. There are uniquely determined integers $\chi_i (X,\mathcal{L})$, $i=0,1,\cdots,n$ such that
\begin{equation*}
\chi_{(X,\mathcal{L})} (k) = \sum_{i=0} ^{n} \chi_i (X,\mathcal{L}) {{n+i-1} \choose {i}}.
\end{equation*}
The degree of $X$ is equal to $\chi_n (X,\mathcal{L})$. The $\Delta$-genus and the sectional genus of $(X,\mathcal{L})$ are defined respectively by the formulas
\begin{equation*}
\begin{CD}
& \Delta (X,\mathcal{L}) & \quad = \quad  & n + \chi_n (X,\mathcal{L}) - h^0 (X,\mathcal{L}); \\
 &g (X,\mathcal{L})      & \quad = \quad  & 1 - \chi_{n-1} (X,\mathcal{L}).
\end{CD}
\end{equation*}
Note that the sectional genus $g (X,\mathcal{L})$ is equal to the arithmetic genus of a generic linear curve section of $X \subset \P^r$.

\section{The Main Theorem}
Throughout this section, $X \subset \P^r$ be an $n$-dimensional irreducible projective variety of codimension at least $2$ which satisfies property $N_p$ for some $p \geq 2$, $q \in \P^r$ is a closed point outside of $X$, $\Sigma_q (X)$ is the secant locus of $X$ with respect to $q$, and
\begin{equation*}
X_q=\pi _q (X) \subset \P^{r-1}.
\end{equation*}
where $\pi_q : X \rightarrow \P^{r-1}$ is a linear projection of $X$ from $q$.

Our purpose in this section is to prove {\bf Theorem 3.1 and Theorem 3.3}.

\begin{theorem}\label{thm:main1}
Let $S$ be the homogeneous coordinate ring of $\P^{r-1}$, $S_{X_q}$ the homogeneous coordinate ring of $X_q \subset \P^{r-1}$, and
\begin{equation*}
E:=  \bigoplus_{j \in \Z} H^0
(X,\mathcal{O}_{X} (j))
\end{equation*}
the graded ring of twisted global sections of $\mathcal{O}_{X} (1)$. \\
$(1)$ If $\Sigma_q (X) = \emptyset$, then there is an exact sequence
\begin{equation*}
0 \rightarrow S_{X_q} \rightarrow E \rightarrow S_{\Lambda} (-1) \rightarrow 0.
\end{equation*}
$(2)$ If $\Sigma_q (X) \neq \emptyset$, then there is an exact sequence
$$ \begin{cases} 0 \rightarrow S_{X_q} \rightarrow E \rightarrow S_{\Lambda} (-1) \rightarrow 0\\
0 \rightarrow \mathcal{O}_{X_q} \rightarrow (\pi_q)_* \mathcal{O}_{X} \rightarrow \mathcal{O}_{\Lambda} (-1) \rightarrow 0
\end{cases}$$
where $\Lambda \subset \P^{r-1}$ is a linear subspace with the homogeneous coordinate ring $S_{\Lambda}$.
\end{theorem}

As the first application of Theorem \ref{thm:main1}, we reprove a version of a result in \cite{V}.

\begin{corollary}\label{cor:birational}
If $\Sigma_q (X) \neq \emptyset$, then it is a quadric in the linear subspace $<q,\Lambda>$. Therefore $\Lambda = \mbox{Sing}(\pi_q)$. In particular, $\pi_q : X \rightarrow X_q$ is birational.
\end{corollary}

\begin{proof}
The second exact sequence in Theorem \ref{thm:main1} shows that $\Lambda \subset X_q$ since $\pi_* \mathcal{O}_X$ is supported on $X_q$. Also since $X_q \subset \P^{r-1}$ is a nondegenerate proper subvariety, $\Lambda \varsubsetneqq X_q$. Therefore $\mathcal{O}_{X_q}$ and $(\pi_q )_* \mathcal{O}_X $ are isomorphic on the nonempty open subset $X_q \subset \Lambda$ of $X_q$. This completes the proof that $\pi_q : X \rightarrow X_q$ is birational and $\mbox{Sing}(X_q) = \Lambda$. Therefore $\Sigma_q (X) \subset <q,\Lambda>$ and every line in $<q,\Lambda>$ passing through $q$ has a
multiple intersection with $X$. This implies that $\Sigma_q (X)$ should contain a hypersurface
\begin{equation*}
Q \subset <q,\Lambda>
\end{equation*}
of degree $v \geq 2$. Observe that if $v \geq 3$ or $Q$ is a proper subset of $< q,\Lambda>$ then
$X$ admits a tri-secant line, which contradicts to the assumption that $X \subset \P^r$ is cut
out by quadrics. Therefore $v = 2$ and $\Sigma_q (X) = Q$ set-theoretically. Finally, the
homogeneous ideal of ¡×$\Sigma_q (X)$ in $< q,\Lambda>$ is indeed generated by a quadric since
$\Sigma_q (X) = <q,\Lambda> \cap X$ as a scheme.
\end{proof}

Theorem \ref{thm:main1} enables us to prove that the cohomological and local properties of
$X_q \subset \P^{r-1}$ can be precisely determined by $\Sigma_q (X)$:

\begin{theorem}\label{thm:main2}
Let $s$ be the dimension of $\Sigma_q (X)$ where $s=-1$ if $\Sigma_q (X) = \emptyset$. \\
$(1)$ $X_q$ is linearly normal if and only if $\Sigma_q (X) \neq \emptyset$. Therefore
$$\Delta(X_q,\mathcal{O}_{X_q} (1))=\begin{cases} \Delta(X,\mathcal{O}_{X} (1)) & \mbox{if $\Sigma_q (X) = \emptyset$, and}\\
\Delta(X,\mathcal{O}_{X} (1))+1 & \mbox{if $\Sigma_q (X) \neq \emptyset$.} \end{cases}$$
$(2)$ $g(X_q,\mathcal{O}_{X_q} (1))=\begin{cases}g(X,\mathcal{O}_{X} (1)) & \mbox{if $s < n-1$, and}\\ g(X,\mathcal{O}_{X} (1))+1 & \mbox{if $s = n-1$.} \end{cases}$ \\
$(3)$ $X_q$ is $j$-normal for every $j \geq 2$. Therefore
\begin{equation*}
h^0 (\P^{r-1},\mathcal{I}_{X_q /\P^{r-1}} (2))=h^0 (\P^r,\mathcal{I}_{X/\P^r} (2))+s-r.
\end{equation*}
$(4)$ $X_q$ satisfies property $N_{3,p-1}$. In particular, the homogeneous ideal of $X_q \subset \P^{r-1}$ is generated by quadratic and cubic forms.\\
$(5)$ $\mbox{Reg}(X_q)=\mbox{max}\{3,\mbox{Reg}(X)\}$.\\
$(6)$ Assume that $H^i (X,\mathcal{O}_{X} (j))=0$ for $1 \leq i \leq n-1$ and all $j \leq -i$. Then
\begin{equation*}
\mbox{depth}(X_q)=\mbox{min}\{\mbox{depth}(X),s+2\}.
\end{equation*}
$(7)$ Every closed point in $\mbox{Sing}(\pi _q)$ is non-normal point of $X_q$. Therefore
\begin{equation*}
\mbox{Nor}(X_q) = \pi _q (\mbox{Nor}(X) \setminus
\Sigma_q (X)) =\pi _q (\mbox{Nor}(X)) \setminus
\mbox{Sing}(\pi _q).
\end{equation*}
where $\mbox{Nor}(Z)$ denotes the locus of normal points of a variety $Z$. In particular, if $X$
is normal then $\pi_q : X \rightarrow  X_q$ is the normalization of $X_q$. \\
$(8)$ Assume that $X$ is locally Cohen-Macaulay and $\mbox{dim}~\Sigma_q (X) < \mbox{dim}~X -1$. Then the generic point $\eta \in X_q$ of $\mbox{Sing}(\pi _q)$ is a Goto point and
\begin{equation*}
\mbox{CM}(X_q)=S_2 (X_q) = X_q \setminus \mbox{Sing}(\pi _q)
\end{equation*}
where $CM(X_q)$ and $S_2 (X_q)$ denote respectively the locus of Cohen-Macaulay points
and that of $S_2$-points.
\end{theorem}

We now turn to proofs of Theorem \ref{thm:main1} and Theorem \ref{thm:main2}.

\noindent {\bf Proof of Theorem \ref{thm:main1}.} Since $\pi_q : X \rightarrow X_q$ is a finite projective morphism,
\begin{equation*}
E \quad \cong \quad \bigoplus_{j \in \Z} H^0 (X_q,(\pi_q)_* \mathcal{O}_{X} \otimes
\mathcal{O}_{X_q} (j))
\end{equation*}
is a finitely generated graded $S$-module. Consider the minimal free resolution
\begin{equation*}
\cdots \rightarrow \bigoplus_{j \geq 0} S(-i-j)^{\beta_{i,j}}
\rightarrow \cdots \rightarrow \bigoplus_{j \geq 0}
S(-1-j)^{\beta_{1,j}} \rightarrow \bigoplus_{j \geq 0}
S(-j)^{\beta_{0,j}} \rightarrow E \rightarrow 0
\end{equation*}
of $E$ as a graded $S$-module. It is easy to check that $\beta_{0,0} = \beta_{0,1}=1$. Now we claim
that $\beta_{i,j} = 0$ for $0 \leq i \leq p -1$ and all $j \geq 2$. Letting $V = H^0(\P^{r-1},\mathcal{O}_{\P^{r-1}}(1))$ and
$\mathcal{M} = \Omega_{\P^{r-1}} \otimes \mathcal{O}_{\P^{r-1}}(1)$, recall that $\beta_{i,j}$ is equal to $\mbox{dim}_K ~ Ker(\varphi_{i,j})$ where
\begin{equation*}
\varphi_{i,j}:H^1 (\P^{r-1},\bigwedge^{i+1} \mathcal{M} \otimes \pi _*
\mathcal{O}_{X} \otimes \mathcal{O}_{\P^{r-1}} (j-1))
\rightarrow \bigwedge^{i+1} V \otimes H^1 (\P^{r-1},\pi _*
\mathcal{O}_{X} \otimes \mathcal{O}_{\P^{r-1}} (j-1)).
\end{equation*}
is the natural homomorphism induced from the Euler sequence
\begin{equation*}
0 \rightarrow \mathcal{M} \rightarrow V \otimes \mathcal{O}_{\P^{r-1}}
\rightarrow \mathcal{O}_{\P^{r-1}} (1) \rightarrow 0.
\end{equation*}
Now let $\widetilde{V}=\pi ^* V \subset H^0 (X,\mathcal{O}_{X} (1))$ and let $\widetilde{\mathcal{M}} = \pi ^* \mathcal{M}$. Then
we have the following commutative diagram:
\begin{equation*}
\begin{CD}
H^1 (\P^{r-1},\bigwedge^{i+1} \mathcal{M} \otimes \pi _*
\mathcal{O}_{X} \otimes \mathcal{O}_{\P^{r-1}} (j-1)) &
\stackrel{\varphi_{i,j}}{\rightarrow} & \bigwedge^{i+1} V \otimes
H^1 (\P^{r-1},\pi _* \mathcal{O}_{X} \otimes
\mathcal{O}_{\P^{r-1}} (j-1)) \\
\wr \parallel &  & \wr \parallel  \\
H^1 (X,\bigwedge^{i+1} \widetilde{\mathcal{M}} \otimes
\mathcal{O}_{X} (j-1)) &
\stackrel{\psi_{i,j}}{\rightarrow} & \bigwedge^{i+1} \widetilde{V}
\otimes H^1 (X,\mathcal{O}_{X} (j-1))
\end{CD}
\end{equation*}
So we need to show that the maps
\begin{equation*}
\psi_{i,j} : H^1 (X,\bigwedge^{i+1}
\widetilde{\mathcal{M}} \otimes \mathcal{O}_{X} (j-1))
\rightarrow \bigwedge^{i+1} \widetilde{V} \otimes H^1
(X,\mathcal{O}_{X} (j-1))
\end{equation*}
are injective for $0 \leq i \leq p-1$ and all $j \geq 2$. We can prove this by applying Lemma \ref{lem:CKP} to $(X,\mathcal{O}_{X} (1))$ since $X \subset \P^r$ satisfies property $N_p$ and $\widetilde{V} \subset H^0 (X,\mathcal{O}_{X} (1))$ has codimension one. Therefore $E$ admits a minimal free resolution of the form
\begin{equation*}
\cdots \rightarrow S(-2)^{\beta_{1,1}} \rightarrow S \oplus S(-1)
\rightarrow E \rightarrow 0.
\end{equation*}
Let $G$ be the kernel of the surjective homomorphism $S \oplus
S(-1) \rightarrow E$. Then we have the following commutative
diagram with exact rows and columns:
\begin{equation*}
(3.1) \quad \quad \begin{CD}
              & 0              &             & 0 &             & 0       & \\
              & \downarrow     &             & \downarrow &             & \downarrow & \\
0 \rightarrow & I_{X_q / \P^{r-1}} & \rightarrow & S & \rightarrow & S_{X_q} & \rightarrow 0 \\
              & \downarrow     &             & \downarrow &             & \downarrow & \\
0 \rightarrow & G & \rightarrow & S \oplus S(-1) & \rightarrow & E & \rightarrow 0 \\
              & \downarrow     &             & \downarrow &             & \downarrow & \\
0 \rightarrow & H & \rightarrow & S(-1) & \rightarrow & F & \rightarrow 0 \\
               & \downarrow     &             & \downarrow &             & \downarrow & \\
 & 0              &             & 0 &             & 0       & \\
\end{CD}
\end{equation*}
Since $G$ is generated by elements of degree $2$, the homogeneous
ideal $H(1) \subset S$ is generated by linear forms.

If $H(1)$ is the irrelevant ideal $S_{+}$, then $F=K(-1)$. Thus by sheafifying the third row of $(3.1)$, we have an isomorphism $\mathcal{O}_{X_q} \cong (\pi_q )_* \mathcal{O}_X$. In particular, $\pi_q : X \rightarrow X_q$ is an isomorphism and hence $\Sigma_q (X) = \emptyset$.

If $H(1)$ is a proper subset of $S_{+}$, then let $\Lambda \subset \P^{r-1}$ be the linear subspace defined by this ideal. Then
$F \cong S_{\Lambda} (-1)$. This gives the first short exact sequence
\begin{equation*}
(3.2) \quad \quad \quad 0 \rightarrow S_{X_q} \rightarrow E \rightarrow S_{\Lambda} (-1) \rightarrow 0.
\end{equation*}
Also the second exact sequence
\begin{equation*}
(3.3) \quad \quad \quad0 \rightarrow \mathcal{O}_{X_q} \rightarrow \pi _*
\mathcal{O}_{X} \rightarrow \mathcal{O}_{\Lambda} (-1)
\rightarrow 0.
\end{equation*}
is immediately obtained by sheafifying  the second exact sequence $(3.2)$. From $(3.3)$, it is obvious that $\pi_q : X \rightarrow X_q$ fails to be isomorphic exactly along $\Lambda$. This implies that $\mbox{Sing}(X_q) = \Lambda$ and $\Sigma_q (X) \neq \emptyset$. \qed \\

\noindent {\bf Proof of Theorem \ref{thm:main2}.} $(1)$ If $\Sigma_q (X) = \emptyset$, then $\pi_q : X \rightarrow X_q$ is isomorphic and hence $X_q \subset \P^{r-1}$ fails to be linearly normal. Now assume that $\Sigma_q (X) \neq \emptyset$. Since $\mathcal{O}_{X_q} (1)$ is very ample and $\pi _q : X \rightarrow X_q$ fails to be isomorphic, we have
\begin{equation*}
r \leq h^0 (X_q,\mathcal{O}_{X_q} (1)) < h^0
(X,\mathcal{O}_{X} (1))=r+1.
\end{equation*}
This completes the proof that $h^0 (X_q,\mathcal{O}_{X_q} (1))=r$ and hence $X_q \subset \P^{r-1}$ is linearly normal.

\noindent $(2)$ From the exact sequence $(3.3)$,
\begin{equation*}
\chi (X_q,\mathcal{O}_{X_q} (k)) = \chi (X,\mathcal{O}_{X} (k))-\chi (\Lambda,\mathcal{O}_{\Lambda} (k-1)).
\end{equation*}
where $\Lambda \cong \P^s$. If $s < n-1$ then $\chi_{n-1} (X_q,\mathcal{O}_{X_q} (1)) =\chi_{n-1}
(X,\mathcal{O}_{X} (1))$ and hence $g(X_q,\mathcal{O}_{X_q} (1))=g(X,\mathcal{O}_{X} (1))$. On the other hand, if $s=n-1$ then
\begin{equation*}
\chi_{n-1} (X_q,\mathcal{O}_{X_q} (1)) =\chi_{n-1}
(X,\mathcal{O}_{X} (1))-\chi_{n-1}
(\P^{n-1},\mathcal{O}_{\P^{n-1}} (1))=\chi_{n-1}
(X,\mathcal{O}_{X} (1))-1
\end{equation*}
and hence $g(X_q,\mathcal{O}_{X_q} (1))=g(X,\mathcal{O}_{X} (1))+1$.

\noindent $(3)$ The commutative diagram $(3.1)$ gives us the following:
\begin{equation*}
\begin{CD}
              &               &             &                                 &             &   0           &             \\
              &               &             &                                 &             & \downarrow    &             \\
              & 0             &             &                                 &             & I_{X_q /\P^{r-1}}    &             \\
              & \downarrow    &             &                                 &             & \downarrow    &             \\
0 \rightarrow & K             & \rightarrow & S(-2)^{\beta_{1,1}}             & \rightarrow &  G            & \rightarrow 0\\
              & \downarrow    &             & \parallel                       &             & \downarrow    &             \\
0 \rightarrow & L             & \rightarrow & S(-2)^{\beta_{1,1}}             & \rightarrow & H             & \rightarrow 0 \\
              & \downarrow    &             &                                 &             & \downarrow    &             \\
              & I_{X_q /\P^{r-1}}    &             &                                 &             & 0             &             \\
              & \downarrow    &             &                                 &             &               &             \\
              & 0             &             &                                 &             &               &             \\
\end{CD}
\end{equation*}
We can observe the following facts:
\begin{enumerate}
\item[$1.$] $H^1 (\P^{r-1} ,\widetilde{G}(j))=H^1 (\P^{r-1}
,\widetilde{K}(j))=0$ for all $j \in \Z$ by the minimality of the
resolution. \item[$2.$] $H^2 (\P^{r-1} ,\widetilde{K}(j))=0$ for all
$j \in \Z$ by the exact sequence
\begin{equation*}
0 \rightarrow \widetilde{K} \rightarrow S(-2)^{\beta_{1,1}}
\rightarrow \widetilde{G} \rightarrow 0.
\end{equation*}
\end{enumerate}
Now we claim that $L$ is $3$-regular. Assume that $H(1)$ is
generated exactly by $\gamma$ linear forms. Then we get the
following commutative diagram:
\begin{equation*}
\begin{CD}
              &  0           &             &      0                        &             &              &             \\
              & \downarrow   &             &  \downarrow                   &             &      &             \\
              & S(-2)^{\beta_{1,1}-\gamma} & = & S(-2)^{\beta_{1,1}-\gamma}   &             &      &             \\
              & \downarrow    &             &  \downarrow                     &             &     &             \\
0 \rightarrow & L             & \rightarrow & S(-2)^{\beta_{1,1}}             & \rightarrow &  H            & \rightarrow 0\\
              & \downarrow    &             &  \downarrow                     &             & \parallel   &             \\
0 \rightarrow & M             & \rightarrow & S(-2)^{\gamma}                  & \rightarrow & H             & \rightarrow 0 \\
              & \downarrow    &             &   \downarrow                    &             &     &             \\
              & 0             &             &  0                              &             &               &             \\
\end{CD}
\end{equation*}
The third row is a part of the Koszul complex twisted by $-1$ of
the ideal $H(1)$ and hence $M$ is $3$-regular, which completes the
proof that $L$ is $3$-regular. By using the exact sequence $0
\rightarrow \widetilde{K} \rightarrow \widetilde{L} \rightarrow
\mathcal{I}_{X_q /\P^{r-1}} \rightarrow 0$, one can check that
\begin{equation*}
H^1 (\P^r,\mathcal{I}_{X_q /\P^{r-1}} (j))=0 \quad \mbox{for all $j \geq 2$.}
\end{equation*}
Therefore $X_q \subset \P^{r-1}$ is $j$-normal for all $j \geq 2$. Since $X \subset \P^r$ and $X_q \subset \P^{r-1}$ are $2$-normal varieties, we have the following equalities:
$$(3.4) \quad \begin{cases} h^0 (\P^r,\mathcal{I}_{X/\P^r} (2)) & = {{r+2} \choose {2}}-h^0 (X,\mathcal{O}_{X} (2)) \\
h^0 (\P^{r-1},\mathcal{I}_{X_q/\P^{r-1}} (2)) & = {{r+1} \choose {2}}-h^0 (X_q,\mathcal{O}_{X_q} (2))
\end{cases}$$
Also the exact sequence $0 \rightarrow S_{X_q} \rightarrow E \rightarrow S_{\Lambda}(-1) \rightarrow 0$ shows that
\begin{equation*}
(3.5) \quad \quad h^0 (X,\mathcal{O}_{X} (2)) -h^0 (X_q,\mathcal{O}_{X_q} (2)) = s +1.
\end{equation*}
Now the desired equality
\begin{equation*}
h^0 (\P^{r-1},\mathcal{I}_{X_q/\P^{r-1}} (2))=h^0 (\P^r,\mathcal{I}_{X/\P^r} (2))+ s-r.
\end{equation*}
is obtained immediately by $(3.4)$ and $(3.5)$.

\noindent $(4)$ Consider the Koszul cohomology long exact sequence induced by $(3.2)$. Indeed $S_{\Lambda} (-1)$ is $2$-regular and Lemma \ref{lem:CKP} guarantees that $E$ admits a minimal free resolution of the form
\begin{equation*}
\cdots \rightarrow S(-p)^{\beta_{p-1,1}} \rightarrow \cdots \rightarrow S(-2)^{\beta_{1,1}} \rightarrow S \oplus S(-1)
\rightarrow E \rightarrow 0.
\end{equation*}
This gives us the desired vanishing of Koszul cohomology groups of $S_{X_q}$, which concludes the proof of $(4)$.

\noindent $(5)$ We claim that the values of $h^i (X_q , \mathcal{O}_{X_q} (j))$ are given as follows:
\begin{enumerate}
\item[$(a)$] For all $j \geq 1$, $h^0 (X_q ,\mathcal{O}_{X_q} (j))=h^0 (X,\mathcal{O}_{X} (j))-{{s+j-1} \choose {s}}$.
\item[$(b)$] If $1 \leq i \leq s-1$ or $i \geq s+2$, then $h^i (X_q ,\mathcal{O}_{X_q} (j)) = h^i (X,\mathcal{O}_{X} (j))$ for all $j \in \Z$.
\item[$(c)$] If $s=0$, then $h^1 (X_q ,\mathcal{O}_{X_q} (j))=\begin{cases} h^1 (X,\mathcal{O}_{X} (j))+1 & \mbox{for $j \leq -1$, and}\\
h^1 (X,\mathcal{O}_{X} (j)) & \mbox{for $j \geq 0$.} \end{cases}$
\item[$(d)$] If $s \geq 1$ and $j \geq -s+1$, then $h^i (X_q ,\mathcal{O}_{X_q} (j)) = h^i (X,\mathcal{O}_{X} (j))$ for $i =s,s+1$.
\item[$(e)$] Assume that $H^i (X,\mathcal{O}_{X} (j))=0$ for $1 \leq i \leq n-1$ and all $j \leq -i$. Then for all $j \leq -s$,
\begin{equation*}
h^s (X_q ,\mathcal{O}_{X_q} (j))=0  \quad \mbox{and} \quad h^{s+1} (X_q ,\mathcal{O}_{X_q} (j))=h^{s+1} (X,\mathcal{O}_{X} (j))+{{-j} \choose {s}}.
\end{equation*}
\end{enumerate}
Indeed, if $\Sigma_q (X) = \emptyset$ and hence $s=-1$, then the statements are trivial. Now assume that $s \geq 0$. Since $X_q \subset \P^{r-1}$ is projectively normal and $\Lambda = \P^s$, $(a)$ follows from the exact sequence $0 \rightarrow S_{X_q} \rightarrow E \rightarrow S_{\Lambda}(-1) \rightarrow 0$. Also the statements in $(b) \sim (e)$ comes immediately by the cohomology long exact sequence induced from $(3.3)$.

Since $X_q \subset \P^{r-1}$ is not a variety of minimal degree, $\mbox{Reg}(X_q) \geq 3$ (cf. {\bf Theorem 5.1}). Also since $X_q \subset \P^{r-1}$ is $j$-normal for all $j \geq 2$ by $(3)$, its Castelnuovo-Mumford regularity is defined as follows:
\begin{equation*}
\mbox{Reg}(X_q) = \mbox{min}~\{ k \geq 3~|~ H^i (X_q,\mathcal{O}_{X_q} (k-1-i))=0 \quad \mbox{for all $i \geq 1$}\}
\end{equation*}
Similarly, the Castelnuovo-Mumford regularity of $X \subset \P^r$ is defined as follows:
\begin{equation*}
\mbox{Reg}(X) = \mbox{min}~\{ k \geq 2~|~ H^i (X,\mathcal{O}_{X} (k-1-i))=0 \quad \mbox{for all $i \geq 1$} \}
\end{equation*}
Since $(b) \sim (d)$ guarantees that
\begin{equation*}
h^i (X_q,\mathcal{O}_{X_q} (k-1-i)) = h^i (X,\mathcal{O}_{X} (k-1-i))
\end{equation*}
for all $k \geq 3$, the proof is completed.

\noindent $(6)$ If $s=-1$, then $X_q \subset \P^{r-1}$ fails to be linearly normal and hence $\mbox{depth} (X_q)=1$. Now assume that
$s \geq 0$. Then the depth of $X \subset \P^r$ and $X_q \subset \P^{r-1}$ are defined as follows, respectively:
$$ \begin{cases} \mbox{depth}(X) = \mbox{max}~\{ i+1 ~|~ i \geq 1 ~ \mbox{and} ~ \bigoplus_{j \in \Z} H^i
(X,\mathcal{O}_{X} (j))\neq 0 \} \\
\mbox{depth}(X_q) = \mbox{max}~\{ i+1  ~|~ i \geq 1~ \mbox{and} ~
\bigoplus_{j \in \Z} H^i (X_q,\mathcal{O}_{X_q} (j))\neq 0 \}
\end{cases}$$
If $s=0$, then $\bigoplus_{j \in \Z} H^1 (X_q,\mathcal{O}_{X_q} (j))\neq
0$ by $(c)$. Therefore $\mbox{depth} (X_q)=2$ which proves the desired
equality $\mbox{depth} (X_q)=\mbox{min}\{\mbox{depth}(X),s+2\}$. If $s \geq 1$, then $(b)$, $(d)$ and $(e)$ guarantee that
\begin{equation*}
\bigoplus_{j \in \Z} H^i (X_q,\mathcal{O}_{X_q} (j)) \cong
\bigoplus_{j \in \Z} H^i (X,\mathcal{O}_{X} (j))\quad \mbox{for $1 \leq i \leq s$}
\end{equation*}
and
\begin{equation*}
\bigoplus_{j \in \Z} H^{s+1} (X_q,\mathcal{O}_{X_q} (j))\neq 0.
\end{equation*}
Therefore if $\mbox{depth}(X) \leq s+1$, then $\mbox{depth}(X_q) =
\mbox{depth}(X)$ and if $\mbox{depth}(X)
\geq s+2$, then $\mbox{depth}(X_q) = s+2$.

\noindent $(7)$ Recall the following elementary fact is useful to study local properties
of a finite birational morphism of projective varieties:\\

\begin{enumerate}
\item[$(*)$] Let $A$ be an integral domain and let $K(A)$ be the quotient field of $A$. Then $A$ fails to be normal if and only if there exists a subring $B \subset K(A)$ such that $A \subset B$ and $B$ is a finite $A$-module.\\
\end{enumerate}

\noindent Now, let $x \in \mbox{Sing}(\pi)$. Then the ring $(\pi_* \mathcal{O}_{X})_x$ is a finite $\mathcal{O}_{X_q,x}$-module such that $(\pi_* \mathcal{O}_{X})_x / \mathcal{O}_{X,x} \cong \mathcal{O}_{\Lambda,x} \neq 0$ by $(3.3)$. Therefore $\mathcal{O}_{X,x}$ fails to be normal by $(*)$.

\noindent $(8)$ Recall that $\eta \subset X_q$ is said to be a Goto point if $\mbox{dim}~\mathcal{O}_{X_q,\eta} > 1$ and
$$ H^i _{\mathfrak{m}_{X_q,\eta}} (\mathcal{O}_{X_q,\eta})= \begin{cases} 0 & \mbox{if $i \neq 1, \mbox{dim}~\mathcal{O}_{X_q,\eta}$, and}\\
K(\eta) & \mbox{if $i=1$.} \end{cases}$$
In our case, $\mbox{dim}~\mathcal{O}_{X_q,\eta} > 1$ holds since we assume that $s < n -1$. Localizing the exact sequence $(\star)$ at $\eta$, we get the following exact sequence of $\mathcal{O}_{X_q,\eta}$-modules:
\begin{equation*}
0 \rightarrow \mathcal{O}_{X_q,\eta} \rightarrow (\pi _* \mathcal{O}_{X})_{\eta} \rightarrow K(\eta) \rightarrow 0
\end{equation*}
Note that $(\pi _* \mathcal{O}_{X})_{\eta}$ is a Cohen-Macaulay $\mathcal{O}_{X_q,\eta}$-module. Indeed $\mathcal{O}_{X_q,x}$ is Cohen-Macaulay for every $x \in \pi_q ^{-1} (\eta)$ since $X$ is locally Cohen-Macaulay. So the above exact sequence enables us to show that
\begin{equation*}
H^1 _{\mathfrak{m}_{X_q,\eta}} (\mathcal{O}_{X_q,\eta}) \cong K(\eta) \quad \mbox{and} \quad H^i _{\mathfrak{m}_{X_q,\eta}} (\mathcal{O}_{X_q,\eta}) =0
\end{equation*}
for all $i \neq 1,\mbox{dim}~\mathcal{O}_{X_q,\eta}$. As $\eta \in X_q$ is not an $S_2$-point, every $y \in \Lambda$ fails to be an $S_2$-point and a Cohen-Macaulay point of $X_q$.  \qed \\

\begin{lemma}\label{lem:CKP}
Let $X$ be a projective variety and let $\mathcal{L} \in \mbox{Pic}X$ be a very ample line bundle satisfying property $N_p$ for some $p \geq 1$. Then for every base point free subspace $V \subset H^0 (X,\mathcal{L})$ of codimension $t \leq p$ and the kernel $\mathcal{M}_V$ of the evaluation homomorphism $V \otimes \mathcal{O}_X \rightarrow \mathcal{L} \rightarrow 0$,
\begin{equation*}
H^1 (X,\bigwedge^{i+1} \mathcal{M}_V \otimes \mathcal{L}^{j-1}) \rightarrow \bigwedge^{i+1} V \otimes H^1 (X,\mathcal{L}^{j-1})
\end{equation*}
is injective for $0 \leq i \leq p-t$ and all $j \geq 2$.
\end{lemma}

\begin{proof}
See the proof of Theorem 2 in Section 3 of \cite{CKP} where our lemma is shown when $V$ is very ample. But one can easily check that the proof is also available for base point free subsystems.
\end{proof}

We conclude this section by providing some examples where Theorem \ref{thm:main2} is sharp
in the sense that the condition "$p \geq 2$" cannot be weakened.

\begin{ex}\label{ex:birational}
Let $C$ be a hyperelliptic curve of genus $g \geq 2$ and let $f:C \rightarrow \P^1$ be the finite morphism of degree $2$, and $A=f^*
\mathcal{O}_{\P^1} (1)$. Then $\mathcal{L}=A^{g+1}$ defines the linearly normal embedding $C \subset \P^{g+2}$. Also the base point free subspace
\begin{equation*}
V:=f^* H^0 (\P^1,\mathcal{O}_{\P^1} (g+1)) \subset H^0
(C,\mathcal{L})
\end{equation*}
defines the morphism $\varphi_{|V|}:C \rightarrow \P^{g+1}$ which is indeed obtained by the
linear projection of $C\subset \P^{g+2}$ from a point $q \in
\P^{r+1}$ outside of $C$. Since $\varphi_{|V|}(C) \subset \P^{g+1}$
is a rational normal curve of degree $g+1$, the morphism $\pi_q : C \rightarrow \pi_q (C)$ is
not birational.
\end{ex}

\begin{ex}\label{ex:Veronesesurface}
Let $S$ be a smooth surface which admits a double covering
$f:S\rightarrow\P^2$ branched along a smooth quartic curve and let
$A=f^* \mathcal{O}_{\P^2} (1)$. Then $\mathcal{L}=A^2$ defines the linearly normal embedding $S \subset \P^6$ (e.g. Example 10.2.4 in \cite{BeS}). Also the base point free subspace
\begin{equation*}
V:=f^* H^0 (\P^2,\mathcal{O}_{\P^2} (2)) \subset H^0 (S,\mathcal{L}),
\end{equation*}
defines the morphism $\varphi_{|V|}:S \rightarrow \P^5$ which is indeed obtained by the
linear projection of $S\subset \P^6$ from a point $q \in \P^6$ outside
of $S$. Since $\varphi_{|V|}(S) \subset \P^5$ is the
Veronese surface, $\pi_q :S \rightarrow \pi_q (S)$ is not
birational.
\end{ex}

Note that $C \subset \P^{g+2}$ and $S \subset \P^6$ in the previous examples satisfy property $N_p$ if and only if $p \leq 1$ respectively by Theorem 2 in \cite{GL2} and
Theorem 1.3 in \cite{GP}. Therefore the hypothesis "$p \geq 2$" in Corollary \ref{cor:birational} cannot be weakened.

\begin{ex}\label{ex:regularity1}
Let $X \subset \P^r$ be a smooth variety which is a complete intersection of two quadratic hypersurfaces. For a general point $q \in \P^r$, $\pi_q : X \rightarrow X_q$ is birational and $X_q \subset \P^{r-1}$ is a hypersurface of degree $4$. Therefore $\mbox{Reg}(X_q)=4$. Also $\Sigma_q (X)$ is the intersection of $X$ and a hyperplane through $q$ by the double point divisor formula. In particular, $\Sigma_q (X)$ is not
a quadric in a subspace of $\P^r$
\end{ex}

\begin{ex}\label{ex:regularity2}
Let $C \subset \P^{g+2}$ be a linearly normal smooth curve of genus $g$ and of degree $2g+2$ such that $\mathcal{O}_C (1)  = \omega_C \otimes \mathcal{O}_C (D)$ for an effective divisor $D$ of degree $4$. Thus $C \subset \P^{g+2}$ satisfies property $N_2$ if and only if $p \geq 1$ by Theorem 2 in \cite{GL2} sice $D$ defines a $4$-secant $2$-plane $\Lambda$ to $C \subset \P^{g+2}$. For every $q \in \Lambda$ outside of $C$, the vanishing ideal of the projected curve $C_q = \pi_q (C) \subset \P^{g+1}$ cannot be generated by quadratic and cubic equations since $C_q$ admits a $4$-secant line.
\end{ex}

By these two examples, the hypothesis "$p \geq 2$" in Theorem \ref{thm:main2} cannot be drooped.

\section{The secant stratification}
Let $X \subset \P^r$ be an $n$-dimensional projective variety satisfying property $N_2$ such
that $r-n \geq 2$, and $q \in \P^r$ a closed point outside of $X$. Then $\Sigma_q (X)$ is either empty
or else a quadric in a subspace of $\P^r$ (Corollary \ref{cor:birational}) and the cohomological and local
properties of $X_q \subset \P^{r-1}$ can be precisely determined by $\Sigma_q (X)$ (Theorem \ref{thm:main2}). Thus
the problem to understand how the cohomological and local properties of $X_q$ behave
when $q$ runs through $\P^r$ is completely answered by classifying all possible secant loci
and decomposing geometrically the ambient space via this classification. Along this
line, we begin with defining the \textit{secant stratum} of $X$. For $-1 \leq s \leq n$, let the $s$-th
secant strata of $X$, denoted $SL_s(X)$, be the set of all closed points $q \in \P^r$ satisfying
$\mbox{dim} ~\Sigma_q (X) = s$. That is,
\begin{equation*}
SL_s (X)= \{ q \in \P^r ~|~ \mbox{dim}~\Sigma_q (X)=s \}
\end{equation*}
where $\mbox{dim}~\Sigma_q (X)=n$ if $q \in X$. Corollary \ref{cor:birational} guarantees that the $n$-th secant strata $SL_s (X)$ is $X$.
The ambient space is expressed as
\begin{equation*}
\P^r = \bigcup_{-1 \leq s \leq n} SL_s (X),
\end{equation*}
which we call the \textit{secant stratification} of $X \subset \P^r$. Recall that $\pi_q : X \rightarrow X_q$ is a singular projection if and only if $q$ is contained in the union of the tangent variety
\begin{equation*}
\mbox{Tan}~(X) = \bigcup_{x \in X} T_x X \subset \P^r
\end{equation*}
of $X$ and the secant variety
\begin{equation*}
\mbox{Sec}~(X) = \overline{\bigcup_{x_1 \neq x_2 , x_i \in X} <x_1 , x_2 >} \subset \P^r
\end{equation*}
to $X$. Thus the stratification
\begin{equation*}
\bigcup_{0 \leq s \leq n} SL_s (X)= \mbox{Tan}~(X) ~\bigcup ~\mbox{Sec}~(X)
\end{equation*}
gives a finer information on simple projections of $X$.

We begin by mentioning two basic observations:

\begin{lemma}\label{lem:positivedimension}
Let $X \subset \P^r$ be a projective variety satisfying property $N_2$ and let $q \in \P^r$ be a closed point outside of $X$. If $\mbox{dim}~\Sigma_q (X) >0$, then $q \in \mbox{Tan}X$.
\end{lemma}

\begin{proof}
The secant locus $\Sigma_q (X)$ is a quadric of a positive dimension in its span. Thus there is a point $z \in \Sigma_q (X)$ such that
\begin{equation*}
<q,z> \quad \subset T_z \quad \Sigma_q (X) \quad  \subset \quad T_z X
\end{equation*}
since $q$ is contained in the span of $\Sigma_q (X)$. Therefore $q \in \mbox{Tan}X$.
\end{proof}

This implies that for every $q \in \mbox{Sec}~ (X)$ outside of $\mbox{Tan}~(X)$, the secant locus of $X$ at $q$ is the union of two simple points, or equivalently, the number of secant lines to $X$ passing through $q$ is equal to one.

\begin{proposition}\label{prop:positivity}
Let $X \subset \P^r$ be a projective variety which satisfies property $N_2$, $L = \mathcal{O}_{X} (1)$, and $q \in \P^r$ a closed point outside of $X$ such that $\Sigma_q (X) \neq \emptyset$.\\
$(1)$ If $L \otimes A^{-2}$ is nef for an ample line bundle $A \in \mbox{Pic}X$, then $\Sigma_q (X)$ is either a finite scheme of length $2$ or else a smooth plane conic curve. \\
$(2)$ If $L \otimes A^{-3}$ is nef for an ample line bundle $A \in \mbox{Pic}X$, then $\Sigma_q (X)$ is a finite scheme of length $2$. In this case, the secant stratification of $X$ is
\begin{equation*}
\P^r = SL_{-1} (X) \cup SL_0 (X) \cup SL_n (X)
\end{equation*}
where $SL_0 (X) = (\mbox{Sec}X \cup \mbox{Tan}X)\setminus X$ and $SL_n (X) = X$.
\end{proposition}

\begin{proof}
Since $\Sigma_q (X)$ is a quadric in a subspace of $\P^r$, $X$ has a line or a smooth plane conic curve if $\mbox{dim}~\Sigma_q (X) \geq 1$, and it has a line if $\mbox{dim}~\Sigma_q (X) \geq 2$.\\
$(1)$ Since $L \otimes A^{-2}$ is nef, $X \subset \P^r$ contains no lines. Therefore $\mbox{dim}~\Sigma_P (X) \leq 1$ and the equality holds if and only if $\Sigma_q (X)$ is a smooth conic curve. \\
$(2)$ Since $L \otimes A^{-3}$ is nef, $X \subset \P^r$ contains no lines and no smooth plane conic curves. Therefore $\mbox{dim}~\Sigma_q (X)=0$.
\end{proof}

From now on, we assume that $K$ is an algebraically closed field of characteristic zero. The remaining part of this section is devoted to obtain the secant stratification of Veronese embeddings and Segre embeddings.

\subsection{Veronese embedding} For $n \geq 2$ and $d \geq 2$, let
\begin{equation*}
X = \nu_d (\P^n) \subset \P^N , \quad N= {{n+d} \choose {n}}-1,
\end{equation*}
be the $d$-uple Veronese embedding of $\P^n$. Note that Theorem \ref{thm:main2} can be applied to $X$ since it satisfies property $N_d$ by Theorem 2.2 in \cite{Gr}.

The secant loci of $X$ are as follows:

\begin{theorem}\label{thm:Veronese}
Under the situation just stated, let $q \in \mbox{Sec}~(X)$ be a closed point outside of $X$. \\
$(1)$ When $d=2$, $\Sigma_q (X)$ is a smooth plane conic curve. Therefore $\mbox{Sec}~(X) = \mbox{Tan}~(X)$ and $\mbox{dim}~\mbox{Tan}~(X)=2n-1$. The secant stratification of $X$ is
\begin{equation*}
\P^N = SL_{-1} (X) \cup SL_{1} (X) \cup X
\end{equation*}
where $SL_{1} (X) = \mbox{Sec}~(X) \setminus X$ and $SL_n (X) = X$.\\
$(2)$ When $d \geq 3$,$\Sigma_q (X)$ is the union of distinct two points if $q \in \mbox{Sec}~(X) \setminus \mbox{Tan}~(X)$ and a double point if $q \in \mbox{Tan}~(X)$. Then the secant stratification of $X$ is
\begin{equation*}
\P^N = SL_{-1} (X) \cup SL_{0} (X) \cup X
\end{equation*}
where $SL_{0} (X) = \mbox{Sec}~(X) \setminus X$ and $SL_n (X) = X$.
\end{theorem}

\begin{proof}
$(1)$ Proposition \ref{prop:positivity}.(1) guarantees that the dimension of $\Sigma_q (X)$ is at most one. So it suffices to show that $\Sigma_q (X)$ contains a smooth plane conic curve. Let $l \subset \P^N$ be a secant line to $X$ which passes through $q$.

When $l \subset T_x X$ for some $x \in X$, let $z \in \P^n$ be such that $\nu_2 (z) =x$. Then by the isomorphism $T_x X \cong T_z \P^n$, $l$ comes from a line $l' \subset \P^n$ which passes through $z$. Thus $C := \nu_2 (l' )$ is a smooth plane conic curve which passes through $x$ such that $T_x C = l$. In particular, $q$ is contained in the plane spanned by $C$ and hence $C \subset \Sigma_q (X)$.

When $X \cap l$ consists of two distinct points $x$ and $y$, let $z$ and $w$ be two closed points in $\P^n$ such that $\nu_2 (z) = x$ and $\nu_2 (w) = y$. Then the line $l' = <z,w>$ maps to a smooth plane conic curve $C$ by $\nu_2$. Therefore $q$ is contained in the plane spanned by $C$ and hence $C \subset \Sigma_q (X)$.\\
$(2)$ The assertion comes immediately by Proposition \ref{prop:positivity}.(2).
\end{proof}

\subsection{Segre embedding} For $a \geq 1$ and $b \geq 2$, let
\begin{equation*}
X = \sigma (\P^a \times \P^b) \subset \P^N, \quad N=ab+a+b,
\end{equation*}
be the Segre variety. When $a=1$, $X$ is a smooth rational normal scroll and hence
it satisfies property $N_p$ for all $p \geq 1$. Also for $a, b \geq 2$, it is shown by A. Lascoux\cite{L} and P. Pragcz and J. Weyman\cite{PW} that $X$ satisfies property $N_p$ if and only if $p \leq 3$. 

The secant loci of $X$ are as follows:

\begin{theorem}\label{thm:Segre}
Under the situation just stated, let $q \in \mbox{Sec}~(X)$ be a closed point outside of $X$. Then $\Sigma_q (X)$ is a smooth quadric surface. Therefore $\mbox{Sec}~(X) = \mbox{Tan}~(X)$ and $\mbox{dim}~\mbox{Tan}~(X)=2(a+b)-1$. The secant stratification of $X$ is
\begin{equation*}
\P^N = SL_{-1} (X) \cup SL_{2} (X) \cup SL_{a+b} (X)
\end{equation*}
where $SL_{2} (X) = \mbox{Sec}~(X) \setminus X$ and $SL_{a+b} (X)=X$.
\end{theorem}

\begin{proof}
By Corollary \ref{cor:birational}, $\Sigma_q (X)$ is a quadric in a subspace $\P^m \subset \P^{ab+a+b}$. So we may
assume that $\Sigma_q (X) \subset \P^m$ is defined by $X_0 ^2 + \cdots + X_k ^2$ for some $k \geq 0$. In this situation
we need to show that $m=k=3$. Let $f : X \rightarrow \P^a$ and $g : X \rightarrow \P^b$ be natural
projection morphisms. For $\lambda \in \P^a$ and $\mu \in \P^b$, $f^{-1} (\lambda) \cong \P^b$ and $g^{-1} (\mu) \cong \P^a$ are
called rulings of $X$. Recall that any linear subspace $\Lambda \subset X$ is contained in a ruling of $X$. Also
distinct two rulings are disjoint or they meet at one point.

The proof will proceed in several steps.\\

\textit{Step 1.} We first show that $\mbox{Sec}~(X) = \mbox{Tan}~(X)$. Let $q \in \mbox{Sec}~(X)$ be a closed point outside of $X$ such that $q \in <x_1,x_2>$ where $x_i=(\lambda_i,\mu_i) \in X$ and $x_1 \neq x_2$. Since $q \notin X$, $\lambda_1 \neq \lambda_2$ and $\mu_1 \neq \mu_2$. For $y  = (\lambda_1,\mu_2)$, let
\begin{equation*}
l_1 = <x_1,y> \subset f^{-1} (\lambda_1) \quad \mbox{and} \quad l_2 = <x_2,y> \subset g^{-1}(\mu_2).
\end{equation*}
Then $Q_{x_1,x_2}=l_1 \times l_2 \subset X$ is a smooth quadratic surface and the three dimensional linear space $<Q_{x_1,x_2}>$ contains $q$. Thus $Q_{x_1,x_2} \subset \Sigma_q (X)$ and $q \in \mbox{Tan}~(X)$. This concludes the proof that $\mbox{Sec}~(X) = \mbox{Tan}~(X)$.\\

\textit{Step 2.} Let $q \in \mbox{Sec}~(X)$ be a closed point outside of $X$. By \textit{Step 1}, we may assume that there exists a closed point $x=(\lambda,\mu) \in X$ such that the line $l = <q,x>$ is tangential to $X$ at $x$. Let $\Lambda_1 = f^{-1} (\lambda)$ and $\Lambda_2 = g^{-1} (\mu)$ be the rulings over $\lambda$ and $\mu$, respectively. Then $l_1 = \Lambda_1 \cap <q,\Lambda_2>$ and $l_2 = \Lambda_2 \cap <q,\Lambda_1>$
are lines passing through $x$ and $l_1 \times l_2 \subset X$ is a smooth quadric surface in the linear subspace $<l_1 , l_2 > = \P^3 \subset \P^N$. Observe that $<l_1 , l_2 >$ is equal to $<q,\Lambda_1 > \cap <q , \Lambda_2 >$. Therefore $q \in < l_1 , l_2 >$ and hence $l_1 \times l_2 \subset \Sigma_q (X)$. This completes the proof that $\Sigma_q (X)$ is not contained in a ruling of $X$ and $m \geq 3$ with equality if and only if $\Sigma_q (X) = l_1 \times l_2$.\\
 
\textit{Step 3.} If $\Sigma_q (X)$ is reducible(i.e. $k=0$ or $1$), then $\Sigma_q (X)$
should be contained in a ruling of $X$ since $m \geq 3$. This contradicts to \textit{Step 2}. Thus $k \geq 2$ and
$\Sigma_q (X)$ is irreducible.\\

\textit{Step 4.} Assume that $k \geq 4$. Then $\mbox{Pic}~\Sigma_q (X)$ is generated by
$\mathcal{O}_{\Sigma_q (X)} (1)$. Note that $f ^* \mathcal{O}_{\P^a} (1) |_{\Sigma_q (X)}$ and $g ^*
\mathcal{O}_{\P^b} (1)|_{\Sigma_q (X)}$ are globally generated and nontrivial since $\Sigma_q (X)$ is not contained in a ruling of $X$. Therefore
\begin{equation*}
f ^* \mathcal{O}_{\P^a} (1) |_{\Sigma_q (X)} =
\mathcal{O}_{\Sigma_q (X)} (u) ~ \mbox{and} ~ g ^*
\mathcal{O}_{\P^b} (1)|_{\Sigma_q
(X)}=\mathcal{O}_{\Sigma_q (X)} (v)
\end{equation*}
for some $u,v \geq 1$. Then $\mathcal{O}_{\Sigma_q (X)} (1) = \{f^* \mathcal{O}_{\P^a} (1)+ g^* \mathcal{O}_{\P^b} (1)\}|_{\Sigma_q (X)} =
\mathcal{O}_{\Sigma_q (X)} (u+v) $ which is a contradiction. This concludes that $2 \leq k \leq 3$.\\

\textit{Step 5.} Assume that $m-k \geq 1$ and hence $\Sigma_q (X)$ is singular. Let $\Delta$ be the vertex
of $\Sigma_q (X)$. Thus $\Delta$ is a linear space of dimension $m-k-1$. Also $\Sigma_q (X)$ has a smooth
quadric $Q \subset \P^k$ such that $\Sigma_q (X)=\mbox{Join} (\Delta,Q)$. For each $z \in
Q$, $<\Delta,z>$ is contained in a ruling $R_z$. If $m-k \geq 2$ and hence $\Delta$ has a positive dimension,
then $\Sigma_q (X)$ is contained in a ruling since $\Delta \subset R_z$ for every $z \in Q$. This contradicts to
\textit{Step 1}. Therefore $m=k$ or $m=k+1$ and $k=2$ or $3$. By \textit{Step 2}, $(k,m)$ should be $(3,3)$ or $(3,4)$. If $(k,m)=(3,4)$, then $\Delta=\{y\}$ is a point and $Q$ is covered by two families of lines $\{L_{\alpha}\}_{\alpha \in \P^1}$ and $\{M_{\beta}\}_{\beta \in \P^1}$. Fix $\alpha \in \P^1$  and let
$R$ be the ruling which contains $<y,L_{\alpha}>$. Then $<y,M_{\beta}> \subset R$ for every $\beta$ since $<y,L_{\alpha}
\cap M_{\beta}> \subset R$. Therefore $\Sigma_q (X) \subset R$ which contradicts to \textit{Step 2}.\\

\noindent By \textit{Step 1} $\sim$ \textit{Step 5}, it is proved that $(m,k)=(3,3)$ and $\Sigma_q (X) = l_1
\times l_2$. Since $\mbox{dim}~\Sigma_q (X)=2$ for every $q \in \mbox{Tan}~(X)$ outside of $X$, $\mbox{dim}~\mbox{Tan}~(X)=2(a+b)-1$.
\end{proof}

\section{Varieties of low degree}
As an application of Theorem \ref{thm:main2}, this section is devoted to reprove some results in \cite{BS} and \cite{HSV} on cohomological structure of varieties of almost minimal degree.

It is well-known that every nondegenerate irreducible projective variety $X \subset \P^r$ satisfies the condition
\begin{equation*}
\mbox{deg}(X) = \mbox{codim}(X,\P^r)+k
\end{equation*}
for some $k \geq 1$. A classical problem in algebraic geometry is to classify and to find a
structure theory of projective varieties that have small $k$ values.

Varieties with $k=1$ are called \textit{varieties of minimal degree}, which were completely classified more than one hundred years ago by P. Del Pezzo and E. Bertini (cf. \cite{EH}, \cite{Fuj}, \cite{Ha}): A variety $X$ of minimal degree is either a quadric hypersurface, (a cone over) the Veronese surface in $\P^5$, or a rational normal scroll. In particular, these varieties are arithmetically Cohen-Macaulay. Also varieties of minimal degree were
characterized cohomologically and homologically:

\begin{theorem}[\cite{EG} and \cite{EGHP}]\label{thm:characterization}
Let $X \subset \P^r$ be a nondegenerate irreducible projective variety. Then the followings are equivalent:
\begin{enumerate}
\item[$(i)$] $X$ is a variety of minimal degree.
\item[$(ii)$] $X$ is $2$-regular.
\item[$(iii)$] $X$ satisfies property $N_{2,p}$ for $p = \mbox{codim}(X,\P^r)$.
\end{enumerate}
\end{theorem}
 
A variety $X \subset \P^r$ with $\mbox{deg}(X)=\mbox{codim}(X,\P^r)+2$ is \textit{a variety of almost minimal degree} due to M. Brodmann and P. Schenzel. 
If furthermore $X$ is linearly normal and $\Delta (X,\mathcal{O}_X (1))=g(X,\mathcal{O}_X (1))=1$, then it is called \textit{a Del Pezzo variety}. They have been extensively studied over the last thirty years (\cite{BS}, \cite{Fuj}, \cite{HSV}, \cite{P1}, \cite{P2}, etc). T. Fujita\cite{Fuj} shows that $X$ is a variety of almost minimal degree if and only if it is either a normal Del Pezzo variety or else the image of a variety $\widetilde{X} \subset \P^{r+1}$ of minimal degree via a projection from a closed point outside of $\widetilde{X}$. 

Theorem \ref{thm:main2} gives a new proof of Theorem A in \cite{HSV} and Theorem 1.3 in \cite{BS}:
 
\begin{corollary}\label{cor:almostminimal}
Let $X \subset \P^r$ be a variety of almost minimal degree such that the codimension $\mbox{codim}(X,\P^r)$ is at least two. Then\\
$(1)$ $\mbox{Reg}(X)=3$.\\
$(2)$ Assume that $X = \pi_q (\widetilde{X})$ where $\widetilde{X} \subset \P^{r+1}$ is a variety of minimal degree and $q \in \P^{r+1}$ is a closed point outside of $\widetilde{X}$. Then 
\begin{equation*}
\mbox{depth}(X)=\mbox{dim}~\Sigma_q (\widetilde{X})+2.
\end{equation*}
\end{corollary}

\begin{proof}
$(1)$ By Fujita's classification theory, $X$ is either a normal Del Pezzo variety or else the image of a variety of minimal degree via a simple projection.

A normal Del Pezzo variety is arithmetically Cohen-Macaulay(cf. I.(3.5) in \cite{Fuj}) and hence the graded Betti diagram is preserved in passing from $X \subset \P^r$ to its general linear curve section $C \subset \P^{r-n+1}$ where $n$ is the dimension of $X$. In particular, $\mbox{Reg}(X) = \mbox{Reg}(C)$. Since $C$ is an elliptic normal curve, it is $3$-regular.
 
When $X = \pi_q (\widetilde{X})$ for a variety $\widetilde{X} \subset \P^{r+1}$ of minimal degree and a closed point $q \in \P^{r+1}$ outside of $\widetilde{X}$, Theorem \ref{thm:main2}.(4) shows that $\mbox{Reg}(X)=3$ since $\widetilde{X} \subset \P^{r+1}$ is $2$-regular.\\
$(2)$ This comes immediately from the depth formula in Theorem \ref{thm:main2}.$(5)$ since $\widetilde{X}$ is arithmetically Cohen-Macaulay.
\end{proof}

Varieties $X \subset \P^r$ with $\mbox{deg}(X) = \mbox{codim}(X, Pr) + 3$ are not well-understood yet.
Along the program of "classifying low degree varieties by projections of classified
varieties", as suggested in \cite{BS}, an important class of those varieties comes from a
simple projection of varieties of almost minimal degree. Since Del Pezzo varieties of
degree at least $5$ are arithmetically Cohen-Macaulay and satisfy property $N_2$, Theorem
\ref{thm:main2} contributes to the above program as one can see in the following
 
\begin{corollary}\label{cor:DelPezzo}
Let $X \subset \P^r$ be a Del Pezzo variety of degree $d \geq 5$. For a closed point $q \in \P^r$ outside of $X$, let $s=\mbox{dim}~\Sigma_q (X)$ and $X_q = \pi_q (X)$. Then\\
$(1)$ $\Sigma_q (X)$ is a quadric in  an $(s+1)$-dimensional linear subspace of $\P^r$. Therefore $\pi_q : X \rightarrow X_q$ is birational.\\
$(2)$ $X_q \subset \P^{r-1}$ is $3$-regular and
\begin{equation*}
h^0 (\P^{r-1},\mathcal{I}_{X_q /\P^{r-1}} (2))= \frac{d(d-3)}{2}+s-r.
\end{equation*}
$(3)$ $\mbox{depth}(X_q)=s+2$. In particular, $X_q \subset \P^{r-1}$ is arithmetically Cohen-Macaulay if and only if $s=\mbox{dim}~X -1$.
\end{corollary}

\begin{proof}
Since $X$ is arithmetically Cohen-Macaulay and its generic linear curve section is a linearly normal curve of arithmetic genus $1$, $X$ satisfies property $N_2$ and $h^0 (\P^r,\mathcal{I}_X (2))=\frac{d(d-3)}{2}$. Therefore the assertions come from Theorem \ref{thm:main2}.
\end{proof}
   
\begin{ex}\label{ex:Grassmann}
Let $X = \G (1,4) \subset \P^9$ be the $6$-dimensional Grassmannian manifold of degree $5$. For $3 \leq k \leq 6$, let $X_k \subset \P^{k+3}$ denote a smooth $k$-dimensional linear section of $X$. Then $X_k \subset \P^{k+3}$ is a Del Pezzo manifold of degree $5$. Let $q \in \P^{k+3}$ be a closed point outside of $X_k$. We claim the followings:\\

\begin{enumerate}
\item[$(a)$] $\mbox{dim}~\Sigma_q (X_k) = k-2$
\item[(b)] A minimal free resolution of the vanishing ideal $I$ of $\pi_q (X_k ) \subset \P^{k+2}$ is 
\begin{equation*}
0 \rightarrow R(-5) \rightarrow R(-4)^5 \rightarrow R(-3)^5 \rightarrow I \rightarrow 0
\end{equation*}
where $R$ is the homogeneous coordinate ring of $\P^{k+2}$.
\end{enumerate}

Indeed Theorem \ref{thm:main2}.(2) implies that
\begin{equation*}
(5.1) \quad \quad h^0 (\P^{k+2},\mathcal{I}_{\pi_q (X_k )} (2))= 2+s-k , \quad s= \mbox{dim}~\Sigma_q (X_k),
\end{equation*}
since $X_k \subset \P^{k+3}$ is cut out by $5$ quadrics. This shows that $s \geq k-2$. On the other hand, note that the Picard group of $X_k$ is generated by $\mathcal{O}_{X_k}(1)$. Thus $X_k$ has no divisors of degree $2$. This shows that $\mbox{dim}~\Sigma_q (X_k) \leq k-2$ since $\Sigma_q (X_k )$ is a quadric in a subspace of $\P^{k+3}$. This completes the proof of $(a)$.

Since it is shown that $\mbox{dim}~\Sigma_q (X_k ) = k-2$, there is no quadric equations in $I$ by the formula $(5.1)$. Also $\mbox{Reg}(\pi_q (X_k ))=3$ and $\mbox{depth} =k$ by Corollary \ref{cor:DelPezzo}. Thus a minimal free resolution of $I$ is of the form
\begin{equation*}
0 \rightarrow R(-5)^{\beta_{3,2}} \rightarrow R(-4)^{\beta_{2,2}} \rightarrow R(-3)^{\beta_{1,2}} \rightarrow I \rightarrow 0.
\end{equation*}
Recall that the graded Betti diagram is preserved in passing from $\pi_q (X_k )$ to its general
linear curve section $C \subset \P^3$. Since
\begin{equation*}
g (X_k , \mathcal{O}_{X_k} (1)) =  g (\pi_q (X_k) , \mathcal{O}_{\pi_q (X_k)} (1)) = 1
\end{equation*}
by Theorem \ref{thm:main2}, $C \subset \P^3$ is an isomorphic projection of an elliptic normal curve $\widetilde{C} \subset \P^4$ of degree $5$. Thus one can compute the Betti numbers $\beta_{1,2}$, $\beta_{2,2}$ and $\beta_{3,2}$.

The secant stratification of $X_k \subset \P^{k+3}$ is 
\begin{equation*}
\P^{k+3} = SL_{k-2} (X_k ) \cup SL_k (X_k )
\end{equation*}
by $(a)$ where $SL_{k-2} (X_k ) = \P^{k+3} \setminus X_k$ and $SL_k (X_k ) = X_k $. 
\end{ex}

\end{document}